\documentclass[12pt,reqno]{amsart}

\usepackage{amsthm, mathrsfs,amssymb,amsmath}
\usepackage{enumerate}
\usepackage[hidelinks]{hyperref}
\usepackage{xcolor}
\hypersetup{
	colorlinks,
	linkcolor={red!50!black},
	citecolor={green!50!black},
	urlcolor={red!80!black}
}
\makeatletter
\@namedef{subjclassname@2010}{%
	\textup{2010} Mathematics Subject Classification}
\makeatother

\allowdisplaybreaks

\newtheorem{thm}{Theorem}[section]

\newtheorem{cor}[thm]{Corollary}

\theoremstyle{definition}
\newtheorem{defn}[thm]{Definition} 

\theoremstyle{remark}
\newtheorem{rem}{Remark}

\newcommand{\spt}{\textnormal{supp}}

\def\R2n{{\mathbb R}^{2n}}

\def\R^2{{\mathbb R}^2}
\def\R2n{{\mathbb R}^{2n}}
\def\R{{\mathbb R}}

\title[Multipliers on commutative hypergroups]{$L^p$-$L^q$ Multipliers on commutative hypergroups}
\author{Vishvesh Kumar}
\address{Vishvesh Kumar \endgraf
	Department of Mathematics: Analysis, Logic and Discrete Mathematics
	\endgraf
	Ghent University, Belgium}
\endgraf
\email{vishveshmishra@gmail.com}
\author[M. Ruzhansky]{Michael Ruzhansky}
\address{
	Michael Ruzhansky:
	\endgraf
	Department of Mathematics: Analysis, Logic and Discrete Mathematics
	\endgraf
	Ghent University, Belgium
	\endgraf
	and
	\endgraf
	School of Mathematical Sciences
	\endgraf
	Queen Mary University of London
	\endgraf
	United Kingdom
	\endgraf
	{\it E-mail address} {\rm michael.ruzhansky@ugent.be}
}

\begin{document}
	
	\begin{abstract} The main purpose of this paper is to prove H\"ormander's  $L^p$-$L^q$ boundedness of Fourier multipliers on commutative hypergroups. We carry out this objective by establishing Paley inequality and Hausdorff-Young-Paley inequality for commutative hypergroups. We show the  $L^p$-$L^q$ boundedness of the spectral multipliers for the generalised radial Laplacian by examining our results on Ch\'{e}bli-Trim\`{e}che hypergroups. As a consequence, we obtain embedding theorems and time asymptotics for the $L^p$-$L^q$ norms of the heat kernel for generalised radial Laplacian. Finally,  we present applications of the obtained results to study the well-posedness of nonlinear partial differential equations.       
	\end{abstract}
	\maketitle
	
	\section{Introduction}  
	\noindent  In this paper we establish $L^p$-$L^q$ boundedness of Fourier multipliers on commutative hypergroups. 
	As an application we will prove the  boundedness result for spectral multipliers in the context of  Ch\'{e}bli-Trim\`{e}che hypergroups. Bloom and Xu \cite{BX99} proved the  H$\ddot{\text{o}}$rmander multiplier theorem concerning the $L^p$-boundedness of Fourier mutlipliers on Ch\'{e}bli-Trim\`{e}che hypergroups by using the techniques developed by Anker  in his seminal paper \cite{Anker}. In particular, they extended the result of Stempak \cite{Stempak} proved for Bessel-Kingman hypergroups, a particular class of Ch\'{e}bli-Trim\`{e}che hypergroups of polynomial growth. After that, Gosselin and Stempak \cite{GStempak} proved an improved version of result in \cite{Stempak} using similar method used by H$\ddot{\text{o}}$rmander in his classical paper \cite{Hormander1960} for Bessel-Kingman hypergroups, which was later studied by several researcher in many different settings, e.g. \cite{Soltani, DPW, DHejna, MOmri, CBO}. Here we deal with $L^p$-$L^q$ multipliers as opposed to the $L^p$-multipliers for which theorems of Mihlin-H$\ddot{\text{o}}$rmander or Marcinkiewicz type provide results for Fourier multipliers in different settings based on the regularity of the symbol. In the Lie group setting, Mihlin-H$\ddot{\text{o}}$rmander multiplier theorem was studied by many researchers. We cite here \cite{ Hormander1960, Anker, Anker92, Anker96, Cow3, Cow1, Cow2, Cow4, Ruzwirth,Del, Astengo} to mention a few of them. 
	
	     The Paley-type inequality describes the growth of the Fourier transform of a function in terms of its $L^p$-norm. Interpolating the Paley-inequality with the Hausdorff-Young inequality one can obtain the following H\"ormander's version of the  Hausdorff-Young-Paley inequality,
\begin{equation}\label{3}
    \left(\int\limits_{\mathbb{R}^n}|(\mathcal{F}_{\R^n}f)(\xi)\phi(\xi)^{ \frac{1}{r}-\frac{1}{p'} }|^r\, d\xi \right)^{\frac{1}{r}}\leq \Vert f \Vert_{L^p(\mathbb{R}^n)},\,\,\,1<p\leq r\leq p'<\infty, \,\,1<p<2.
\end{equation} Also, as a consequence  of the Hausdorff-Young-Paley inequality, H\"ormander \cite[page 106]{Hormander1960} proved that the condition 
\begin{equation}\label{4}
    \sup_{t>0}t^b|\{\xi\in \mathbb{R}^n:m(\xi)\geq t\}|<\infty,\quad \frac{1}{p}-\frac{1}{q}=\frac{1}{b},
\end{equation}where $1<p\leq 2\leq q<\infty,$ implies the existence of a bounded extension of a Fourier multiplier $T_{m}$ with symbol $m$ from $L^p(\mathbb{R}^n)$ to $L^q(\mathbb{R}^n).$ Recently, the second author with R. Akylzhanov  studied  classical H\"ormander results for unimodular locally compact groups and homogeneous spaces \cite{ARN, AR}. In \cite{AR}, the key idea behind the extension of H\"ormander theorem is the reformulation of this theorem as follows: 
$$\|T_m\|_{L^p(\mathbb{R}^n) \rightarrow L^q(\mathbb{R}^n)} \lesssim \sup_{s>0} s\left( \int_{ \{\xi \in \mathbb{R}^n :\, m(\xi) \geq s\} } d\xi \right)^{\frac{1}{p}-\frac{1}{q}} \simeq \|m\|_{L^{r, \infty}(\mathbb{R}^n)} \simeq \|T_m\|_{L^{r, \infty}(\textnormal{VN}(\mathbb{R}^n))},$$ where $\frac{1}{r}=\frac{1}{p}-\frac{1}{q},$ $\|m\|_{L^{r, \infty}(\mathbb{R}^n)}$ is the Lorentz norm of $m,$ and $\|T_m\|_{L^{r, \infty}(\textnormal{VN}(\mathbb{R}^n))}$ is the norm of the operator $T_m$ in the Lorentz space on the group von Neumann algebra $\textnormal{VN}(\mathbb{R}^n)$ of $\mathbb{R}^n.$ 
They use the Lorentz spaces and group von Neumann algebra technique for extending it to general locally compact unimodular groups. The unimodularity assumption has its own advantages such as existence of the canonical trace on the group von Neumann algebra and consequently,  Plancherel formula and the Hausdorff-Young inequality. 

In this paper, we prove Paley-type inequality, Hausdorff-Young-Paley inequality and H\"ormander multiplier theorem for commutative hypergroups.  We refer to Section \ref{Jacobinota} for more details and all the notation used here.
The following theorem is an analogue of Paley inequality for commutative hypergroups.

\begin{thm}
	      Suppose that $\psi$ is a positive function  on $\widehat{H}$  satisfying the condition 
	    \begin{equation}
	        M_\psi := \sup_{t>0} t \int_{\underset{\psi(\chi)\geq t}{\chi \in \widehat{H}}} d\pi(\chi) <\infty.
	        \end{equation}
	        Then for  $f \in L^p(H, d\lambda),$ $1<p\leq 2,$ we have 
	        \begin{align}
	            \left( \int_{\widehat{H}} |\widehat{f}(\chi)|^p\, \psi(\chi)^{2-p} d\pi(\chi) \right)^{\frac{1}{p}} \lesssim M_{\psi}^{\frac{2-p}{p}}\, \|f\|_{L^p(H, d\lambda)}.
	        \end{align}
	\end{thm}
The following Hausdorff-Young-Paley inequality is shown for commutative hypergroups. 
\begin{thm} Let $1<p\leq 2,$ and let   $1<p \leq b \leq p' \leq \infty,$ where $p'= \frac{p}{p-1}.$ If $\psi(\chi)$ is a positive function on $\widehat{H}$ such that 
 \begin{equation}
	        M_\psi := \sup_{t>0} t \int_{\underset{\psi(\chi)\geq t}{\chi \in \widehat{H}}} d\pi(\chi)
	        \end{equation}
is finite then for every $f \in L^p(H, d\lambda)$ 
 we have
\begin{equation} 
    \left( \int_{\widehat{H}}  \left( |\widehat{f}(\chi)| \psi(\chi)^{\frac{1}{b}-\frac{1}{p'}} \right)^b d\pi(\chi)  \right)^{\frac{1}{b}} \lesssim M_\varphi^{\frac{1}{b}-\frac{1}{p'}} \|f\|_{L^p(H, d\lambda)}.
\end{equation}
\end{thm}
Next, we establish the following $L^p$-$L^q$ boundedness result for Fourier multipliers on commutative hypergroups.
\begin{thm}   Let $1<p \leq 2 \leq q<\infty$ and let $H$ be a commutative  hypergroup. Suppose that $T$ is a  Fourier multiplier with symbol $h.$ Then we have 
$$\|T\|_{L^p(H,\, d\lambda) \rightarrow L^q(H,\, d\lambda)}\lesssim \sup_{s>0} s \left[ \int_{\{ \chi \in \widehat{H}: |h(\chi)|\geq s\}} d\pi(\chi) \right]^{\frac{1}{p}-\frac{1}{q}}.$$
   \end{thm}
   Now, we establish a spectral multiplier theorem for the operator $L,$ called {\it the generalised radial Laplacian}, given by 
   \begin{equation} \nonumber
	L=L_{A,x}:= -\frac{d^2}{dx^2}- \frac{A'(x)}{A(x)} \frac{d}{dx},
	\end{equation}
	where the function $A,$ is the {\it Ch\'{e}bli-Trim\`{e}che function}. Define $\rho:= \frac{1}{2} \lim\limits_{x \rightarrow \infty} \frac{A'(x)}{A(x)}$. It is known that this limit exists and is non-negative (see Section \ref{sec4.1} for more details and notation).  For some particular choices of  Ch\'{e}bli-Trim\`{e}che function, the operator $L$ arises as the radial part of Laplace-Beltrami operators on symmetric spaces of rank one and on the  Euclidean space. It has been recently observed in \cite{Biswas} that $L$ appeared as the radial part of Laplace-Beltrami operators on simply connected harmonic manifolds of purely exponential volume growth with $A$ as the density function on the harmonic manifold.
	\begin{thm} Let $L$ be the generalised radial Laplacian and $1<p \leq 2 \leq q <\infty.$ Suppose that $\varphi$ is a monotonically  decreasing continuous function on $[\rho^2, \infty)$ such that $\lim_{u \rightarrow \infty}\varphi(u)=0.$ Then we have 
\begin{equation} \nonumber
    \|\varphi(L)\|_{\textnormal{op}} \lesssim \sup_{u>\rho^2} \varphi(u)  \begin{cases} (u-\rho^2)^{(a+1)(\frac{1}{p}-\frac{1}{q})} & \quad \textnormal{if} \quad (u-\rho^2)^{\frac{1}{2}} \leq K, \\ \left[K^{2a+2}- K^{2\alpha+2}+(u-\rho^2)^{(\alpha+1)} \right]^{\frac{1}{p}-\frac{1}{q}} & \quad \textnormal{if} \quad (u-\rho^2)^{\frac{1}{2}}>K,  \end{cases}
\end{equation} where $K, \alpha$ and $a$ are constants appearing in the estimate of the $c$-function,  in the definition of Ch\'ebli-Trim\'eche function $A$ and in the Condition $(P)$ on $A$ respectively, and $\|\cdot\|_{\textnormal{op}}$ denotes the operator norm from $L^p(\mathbb{R}_+, A dx)$ to $L^q(\mathbb{R}_+, Adx).$ 
\end{thm}
The precise values of $K,$ $\alpha$ and $a$ and their appearance in the context of analysis on Ch\'ebli-Trim\'eche hypergroups will be given in Section \ref{sec4.1}.

The above spectral multipliers theorem leads to very interesting results concerning the  $L^p$-$L^q$ norm of the heat kernel of $L$ and embedding theorems for $L,$ by a particular choice of  the function $\varphi.$
\begin{cor} Let $L$ be the generalised radial Laplacian. For any $1<p \leq 2 \leq q <\infty$ there exists a positive constant $C=C_{\alpha, a, p, q, K}$ such that 
 \begin{align*} 
        \|e^{-tL}\|_{\textnormal{op}} \lesssim \begin{cases} t^{-2(\alpha+1)(\frac{1}{p}-\frac{1}{q})} \quad & \text{if}\quad 0<t < \frac{\alpha+1}{K}\left(\frac{1}{p}-\frac{1}{q} \right) \\e^{-t \rho^2}\,\, e^{-\frac{(a+1)^2}{t}\left(\frac{1}{p}-\frac{1}{q}\right)^2} t^{-2(a+1)(\frac{1}{p}-\frac{1}{q})} \quad & \text{if}\quad t \geq  \frac{a+1}{K}\left(\frac{1}{p}-\frac{1}{q} \right),\end{cases}
    \end{align*} where $\|\cdot\|_{\textnormal{op}} $ denotes the operator norm from $L^p(\mathbb{R}_+, A dx)$ to $L^q(\mathbb{R}_+, Adx).$
    
    We also have the embeddings
    $$\|f\|_{L^q(\R_+, A dx)} \leq C \|(1+L)^b f\|_{L^q(\R_+, A dx)}$$ provided that $b \geq (\alpha+1) \left( \frac{1}{p}-\frac{1}{q} \right),\,\,1<p\leq 2 \leq q<\infty.$
    \end{cor}
    Next, we will present applications of our main results in the context of well-posedness of nonlinear abstract Cauchy problems in the space $L^\infty(0, T, L^2(\mathbb{R}_+, A dx)).$ First, we consider the heat equation 
 \begin{equation} \label{Heatinto} 
     u_t-|Bu(t)|^p=0,\quad u(0)=u_0,
 \end{equation}
 where $B$ is a Fourier multiplier from $L^2(\mathbb{R}_+, A dx)$ into $L^{2p}(\mathbb{R}_+, A dx)$ for $1<p<\infty.$ We study local well-posedness of the above heat equation \eqref{Heatinto}. Secondly, we consider  
 the initial value problem  for the nonlinear wave equation  
\begin{align}\label{E-WNLEint}
u_{tt}(t) - b(t)|Bu(t)|^{p} = 0,
\end{align} with the initial condition $
u(0)=u_0, \,\,\, u_t(0)=u_1, $
where $b$ is a positive bounded function depending only on time, $B$ is a Fourier multiplier from $L^2(\mathbb{R}_+, A dx)$ into $L^{2p}(\mathbb{R}_+, A dx)$ for $1<p<\infty.$ We explore the global and local well-posedness of \eqref{E-WNLEint} under some conditions of function $b.$

We organise the paper as follows: in next section we present basics of Fourier analysis on commutative hypergroups. In Section \ref{sec3} will be devoted to the study of $L^p$-$L^q$ Fourier multipliers on commutative hypergroups along with the proof of the key inequalities needed to establish our main results. In Section \ref{sec4}, we analyse of results of previous section in the setting of Ch\'ebli-Trim\'eche hypergroups and then use it to obtained main result of this section on $L^p$-$L^q$ boundedness of spectral multipliers of generalised radial Laplacian and its important consequences. In the last section, we discuss the applications of our result to the well-posedness of nonlinear abstract Chauchy problems.   

Throughout the paper, we shall use the notation $A \lesssim B$ to indicate $A\leq cB $ for a suitable constant $c >0$. 

	\section{Fourier analysis on commutative hypergroups} \label{Jacobinota}
	In this section, we present essentials of the Fourier analysis on commutative hypergroups. 
	 We begin this section with the definition of a hypergroup. In \cite{Jewett}, Jewett refers to hypergroups as convos.
	\begin{defn} A {\it hypergroup} is a non empty locally compact Hausdorff space  $H$  with a weakly continuous, associative convolution $*$ on the Banach space $M(H)$ of all bounded regular Borel measures on $H$ such that $(M(H), *)$ becomes a Banach algebra and the following properties hold: 
		\begin{enumerate}
			\item[(i)] For any $x,y \in H,$ the convolution $\delta_x*\delta_y$ is a probability measure with compact support, where $\delta_x$ is the point mass measure at $x.$ Also, the mapping $(x,y)\mapsto \spt(\delta_x*\delta_y)$ is continuous from $K\times K$ to the space  $\mathcal{C}(H)$ of all nonempty compact subsets of $H$ equipped with the Michael (Vietoris) topology (see \cite{mi} for details).  
			\item[(ii)]  There exists a unique element $e \in H $ such that $\delta_x*\delta_e=\delta_e*\delta_x=\delta_x$ for every $x\in H.$
			\item[(iii)] There is a homeomophism $x \mapsto \check{x}$ on $H$ of order two which induces an involution on $M(H)$ where  $\check{\mu}(E)= \mu (\check{E})$ for any Borel set $E,$ and  $e \in \spt(\delta_x*\delta_y)$ if and only if $x = \check{y}.$ 
		\end{enumerate}
	\end{defn}
	Note that the weak continuity assures that the convolution of bounded measures on a hypergroup is uniquely determined by the convolution of point measures. A hypergroup is called  a commutative hypergroup if the convolution is commutative. A hypergroup $H$ is called {\it hermitian} if the involution on $H$ is the identity map, i.e., $\check{x}=x$ for all $x \in H.$ 
	Note that a hermitian hypergroup is commutative. Every locally compact abelian group is  a trivial example of a commutative hypergroup. Other important examples include Gelfand pairs, Bessel-Kingman hypergroups, Jacobi hypergroups and Ch\'{e}bli-Trim\`{e}che hypergroups.
	
	A {\it left Haar measure} $\lambda$ on $H$ is a non-zero positive Radon measure $\lambda$ such that 
	$$\int_H f(x*y) d\lambda(y)=\int_H f(y)\, d\lambda(y)\quad (\forall x \in H, \, f \in C_c(H)),$$ where we used the notation $f(x*y)=(\delta_x*\delta_y)(f)$. It is well known that a Haar measure is unique (up to scalar multiple) if it exists \cite{Jewett}.  Throughout this article, a left Haar measure is simply called a Haar measure. 
	We would like to make a remark here that it is still not known if a general hypergroup has a Haar measure but several important classes of hypergroups including commutative hypergroups, compact hypergroups, discrete hypergroups, nilpotent hypergroups possess a Haar measure, see  \cite{Jewett, Dunkl, Bloom, Willson, Amini, KKA, KKAadd, KumarRamsey} for more detail on harmonic analysis on hypergroups and its applications.  
	
	Denote the space of all bounded continuous complex valued functions on $H$ by $C^b(H).$ The dual space of a commutative hypergroup $H$ is defined by $$\widehat{H}=\{\chi\in C^b(H):\chi\neq 0,\, \chi(\check{x})=\overline{\chi(x)} \mbox{ and }\chi(x*y)=\chi(x)\chi(y)\ \forall\ x,y\in H\}.$$ The elements of $\widehat{H}$ are called {\it characters} of $H.$ We equip $\widehat{H}$ also with the compact-open topology so that $H$ becomes a locally compact Hausdorff space. Note that $\widehat{H}$ need not possess a hypergroup structure. For example, see \cite[Example 9.3C]{Jewett}. 
 
 The Fourier transform of $f\in L^1(H, \lambda)$ is defined by 
\begin{eqnarray} \label{f}
\widehat{f}(\chi)= \widehat{f}(\chi)=\int_H f(x)\overline{\chi(x)}\ d\lambda(x), \,\,\, \forall\ \chi\in\widehat{H}.
\end{eqnarray}   
  By \cite[Theorem 2.2.4]{Bloom}, the mapping $f \mapsto\widehat{f}$ is a norm-decreasing $*$-algebra homomorphism from $L^1(H, \lambda)$ into $L^\infty(\widehat{H}, \pi).$ Furthermore, $\widehat{f}$ vanishes at infinity. Also, as in the case of locally compact abelian groups, there exists a unique positive Borel measure $\pi$ on $\widehat{H}$ such that $$\int_H|f(x)|^2\ d\lambda(x)=\int_{\widehat{H}}|\widehat{f}(\chi)|^2\ d\pi(\chi) \,\,\, \forall\ f\in L^2(H, \lambda)\cap L^1(H, \lambda).$$ In fact, the Fourier transform extends to a unitary operator from $L^2(H, \lambda)$ onto $L^2(\widehat{H}, \pi).$  We would like to remark here that the support of $\pi,$ denoted $\mathcal{S},$ need not be equal to $\widehat{H}$ (see \cite[Example 2.2.49]{Bloom}). When $\mathcal{S}=\widehat{H},$  $H$ is called  a {\it strong hypergroup}.


	The following theorem is the   Hausdorff-Young inequality for commutative hypergroups \cite{Sinayoung}. 
	\begin{thm} \label{HY}
	Let $p, p'$ be such that  $1 \leq p \leq 2$ and $\frac{1}{p}+\frac{1}{p'}=1.$ Then for $f \in L^2(H,\, d\lambda)\cap L^1(H,\, d\lambda) $ we have the inequality 
	\begin{align*}
	    \|\widehat{f}\|_{L^{p'}(\widehat{H}, d\pi)} \leq \|f\|_{L^p(H,\, d\lambda)}.
	\end{align*}
	\end{thm}
	\section{Fourier multipliers for commutative hypergroups} \label{sec3}
		
	\subsection{Paley  inequality for commutative hypergroups} In this subsection, we prove the Paley inequality for commutative  hypergroups. The Paley inequality has been proved for many different setting, e.g., compact homogeneoups spaces \cite{ARN,ARN1}, compact quantum groups \cite{AMR, Youn}, compact hypergroups \cite{KR} and locally compact groups \cite{AR}. 
	
	Paley inequality is an important inequality in itself but also plays a vital role in obtaining the  Hausdorff-Young-Paley inequality for commutative hypergroups. We follow the strategy of the proof of the Paley inequality in \cite{AR} in the non-commutative setting.

	\begin{thm}\label{Paley} Let $H$ be a commutative hypergroup equipped with a Haar measure $\lambda$ and let $\widehat{H}$ be the dual of $H$ equipped with measure $\pi.$
	      Suppose that $\psi$ is a positive function  on $\widehat{H}$  satisfying the condition 
	    \begin{equation}
	        M_\psi := \sup_{t>0} t \int_{\underset{\psi(\chi)\geq t}{\chi \in \widehat{H}}} d\pi(\chi) <\infty.
	        \end{equation}
	        Then for  $f \in L^p(H, d\lambda),$ $1<p\leq 2,$ we have 
	        \begin{align} \label{Paleyin}
	            \left( \int_{\widehat{H}} |\widehat{f}(\chi)|^p\, \psi(\chi)^{2-p} d\pi(\chi) \right)^{\frac{1}{p}} \lesssim M_{\psi}^{\frac{2-p}{p}}\, \|f\|_{L^p(H, d\lambda)}.
	        \end{align}
	\end{thm}
	\begin{proof}  Let us consider a measure $\nu$ on $\widehat{H}$ given by \begin{equation} \label{nu}
	    \nu(\chi)= \psi(\chi)^2 d\pi(\chi).
	\end{equation}
	We define the corresponding $L^p(\widehat{H}, \nu)$- space, $1 \leq p<\infty,$ as the space of all complex-valued function $f$ defined by $\widehat{H}$ such that 
	$$ \|f\|_{L^p(\widehat{H}, \nu)}:= \left( \int_{\widehat{H}} |f(\chi)|^p \, \psi(\chi)^2 d\pi(\chi) \right)^{\frac{1}{p}}<\infty.$$
	 We define a sublinear operator $T$ for $f \in L^p(H, d\lambda)$ by 
	$$ Tf(\chi):= \frac{|\widehat{f}(\chi)|}{\psi(\chi)} \in L^p(\widehat{H}, \nu).$$
We will show that $T$ is well-defined and bounded from $L^p(H, d\lambda)$ to $L^p(\widehat{H}, \nu)$ for any $1 < p \leq 2.$	
In other words, we claim the following estimate:
\begin{equation} \label{vis10paley}
    \|Tf\|_{L^p(\widehat{H}, \nu)}= \left( \int_{\widehat{H}} \frac{|\widehat{f}(\chi)|^p}{\psi(\chi)^p}\,\psi(\chi)^2 d\pi(\chi) \right)^{\frac{1}{p}} \lesssim M_{\psi}^{\frac{2-p}{p}} \|f\|_{L^p(H, d \lambda)},
\end{equation}
which will give us the required inequality \eqref{Paleyin} with $M_\psi := \sup_{t>0} \int_{\underset{\psi(\chi) \geq t}{\chi \in \widehat{H}}} d\pi(\lambda).$ 
	We will show that $T$ is weak-type $(2,2)$ and weak-type $(1,1).$  More precisely, with the distribution function,  
$$\nu(y; Tf)= \int_{\underset{\frac{|\widehat{f}(\chi)|} {\psi(\chi)} \geq y}{\chi \in \widehat{H}}} \psi(\chi)^2 d\pi(\chi), $$ where $\nu$ is given by formula \eqref{nu}, we show that 
\begin{equation} \label{vish5.4}
    \nu(y; Tf) \leq \left( \frac{M_2 \|f\|_{L^2(H, d\lambda)}}{y} \right)^2 \,\,\,\,\,\,\text{with norm}\,\, M_2=1,
\end{equation}
\begin{equation} \label{vish5.5}
    \nu(y; Tf) \leq \frac{M_1 \|f\|_{L^1(H, d\lambda)}}{y}\,\,\,\,\,\,\text{with norm}\,\, M_1=M_\psi.
\end{equation} 
	Then the estimate \eqref{vis10paley} follows from the Marcinkiewicz interpolation Theorem. Now, to show \eqref{vish5.4}, using Plancherel identity we get
	\begin{align*}
    y^2 \nu(y; Tf) &\leq \sup_{y>0}y^2 \nu(y; Tf)=: \|Tf\|^2_{L^{2, \infty}(\widehat{H}, \nu)}  \leq \|Tf\|^2_{L^2(\widehat{H}, \nu)} \\&= \int_{\widehat{H}} \left( \frac{|\widehat{f}(\chi)|}{\psi(\chi)} \right)^2 \psi(\chi)^2 d\pi(\chi) \\&= \int_{\widehat{H}} |\widehat{f}(\chi)|^2\, d\pi(\chi)  = \|f\|_2^2.  
    \end{align*}
Thus, $T$ is type $(2,2)$ with norm $M_2 \leq 1.$ Further, we show that $T$ is of weak type $(1,1)$ with norm $M_1=M_\psi$; more precisely, we show that 
\begin{align} \label{11weak}
    \nu \left\{ \chi \in \widehat{H}: \frac{|\widehat{f}(\chi)|}{\psi(\chi)}>y \right\} \lesssim M_\psi \frac{\|f\|_{L^1(H, d\lambda)}}{y}.
\end{align}
The left hand side is an integral $ \int \psi(\chi) d\pi(\chi)$ taken over all those $\chi \in \widehat{H}$ for which $\frac{|\widehat{f}(\chi)|}{\psi(\chi)}>y.$
Since $|\widehat{f}(\chi)| \leq \|f\|_1$ for all $\chi \in \widehat{H}$ we have
$$\left\{ \chi \in \widehat{H}: \frac{|\widehat{f}(\chi)|}{\psi(\chi)}>y \right\} \subset \left\{ \chi \in \widehat{H}: \frac{\|f\|_1}{\psi(\chi)}>y \right\},$$ 
 for any $y>0$ and, therefore,
$$\nu \left\{ \chi \in \widehat{H}: \frac{|\widehat{f}(\chi)|}{\psi(\chi)}>y \right\} \leq \nu \left\{ \chi \in \widehat{H}: \frac{\|f\|_1}{\psi(\chi)}>y \right\}.$$

Now by setting $w:=\frac{\|f\|_1}{y},$ we have 
\begin{align}
    \nu \left\{ \chi \in \widehat{H}: \frac{\|f\|_1}{\psi(\chi)}>y \right\} \leq  \int_{\overset{\chi \in \widehat{H}}{\psi(\chi) \leq w} } \psi(\chi)^2\, d\pi(\chi).
\end{align}
Now we claim that 
\begin{align} \label{claim}
    \int_{\overset{\chi \in \widehat{H}}{\psi(\chi) \leq w}} \psi(\chi)^2\, d\pi(\chi) \lesssim M_\psi w.
\end{align}
Indeed, first we notice that 
\begin{align*}
    \int_{\overset{\chi \in \widehat{H}}{\psi(\chi) \leq w} } \psi(\chi)^2\, d\pi(\chi) = \int_{\overset{\chi \in \widehat{H}}{\psi(\chi) \leq w} } \, d\pi(\chi) \int_0^{\psi(\chi)^2} d\tau. 
\end{align*}
By interchanging the order of integration we get 
\begin{align*}
    \int_{\overset{\chi \in \widehat{H}}{\psi(\chi) \leq w} } \,d\pi(\chi) \int_0^{\psi(\chi)^2} d\tau = \int_{0}^{w^2} d\tau \int_{\underset{\tau^{\frac{1}{2}} \leq \psi(\chi) \leq w}{\chi \in \widehat{H}}} d\pi(\chi).
\end{align*}
Further, by making substitution  $\tau= t^2,$ it gives 
\begin{align*}
    \int_{0}^{w^2} d\tau \int_{\underset{\tau^{\frac{1}{2}} \leq \psi(\chi) \leq w}{\chi \in \widehat{H}}} d\pi(\chi) &= 2 \int_0^w t\, dt \int_{\underset{t \leq \psi(\chi) \leq w}{\chi \in \widehat{H}}} d\pi(\chi) \\&\leq 2 \int_0^w t\, dt \int_{\underset{t \leq \psi(\chi) }{\chi \in \widehat{H}}} d\pi(\chi).
\end{align*}
Since 
$$ t \int_{\underset{t \leq \psi(\chi) }{\chi \in \widehat{H}}} d\pi(\chi) \leq \sup_{t>0} t \int_{\underset{t \leq \psi(\chi) }{\chi \in \widehat{H}}} d\pi(\chi) = M_\psi $$ is finite by assumption $M_\psi<\infty,$ we have 
\begin{align*}
    2 \int_0^w t\, dt \int_{\underset{t \leq \psi(\chi) }{\chi \in \widehat{H}}} d\pi(\chi) \lesssim M_\psi w.
\end{align*}
This establishes our claim \eqref{claim} and eventually proves \eqref{11weak}. So, we have proved \eqref{vish5.4} and \eqref{vish5.5}. Then  by using the Marcinkiewicz interpolation theorem with $p_1=1$ and $p_2=2$ and $\frac{1}{p}= 1-\theta+\frac{\theta}{2}$ we now obtain
$$\left( \int_{\widehat{H}} \left(\frac{|\widehat{f}(\chi)|}{\psi(\chi)} \right)^p \psi(\chi)^2\, d\pi(\chi) \right)^{\frac{1}{p}}= \|Tf\|_{L^p(\widehat{H},\, \nu)} \lesssim M_\psi^{\frac{2-p}{p}} \|f\|_{L^p(H, d\lambda)}.$$
This completes the proof of the theorem. \end{proof}

	\subsection{Hausdorff-Young-Paley inequality for commutative hypergroups} In this subsection, we prove the Hausdorff-Young-Paley inequality for commutative  hypergroups. The Hausdorff-Young-Paley inequality is an important inequality in itself but it serves as an essential tool to prove $L^p$-$L^q$ Fourier multiplier for commutative  hypergroups. 
	The following theorem \cite{BL} is useful to  obtain  our result.

\begin{thm} \label{interpolation} Let $d\mu_0(x)= \omega_0(x) d\mu'(x),$ $d\mu_1(x)= \omega_1(x) d\mu'(x),$ and write $L^p(\omega)=L^p(\omega d\mu')$ for the weight $\omega.$ Suppose that $0<p_0, p_1< \infty.$ Then 
$$(L^{p_0}(\omega_0), L^{p_1}(\omega_1))_{\theta, p}=L^p(\omega),$$ where $0<\theta<1, \, \frac{1}{p}= \frac{1-\theta}{p_0}+\frac{\theta}{p_1}$ and $\omega= \omega_0^{\frac{p(1-\theta)}{p_0}} \omega_1^{\frac{p\theta}{p_1}}.$
\end{thm} 

The following corollary is immediate.

\begin{cor}\label{interpolationoperator} Let $d\mu_0(x)= \omega_0(x) d\mu'(x),$ $d\mu_1(x)= \omega_1(x) d\mu'(x).$ Suppose that $0<p_0, p_1< \infty.$  If a continuous linear operator $A$ admits bounded extensions, $A: L^p(Y,\mu)\rightarrow L^{p_0}(\omega_0) $ and $A: L^p(Y,\mu)\rightarrow L^{p_1}(\omega_1).$  Then, there exists a bounded extension $A: L^p(Y,\mu)\rightarrow L^{b}(\omega) $ of $A$, where  $0<\theta<1, \, \frac{1}{b}= \frac{1-\theta}{p_0}+\frac{\theta}{p_1}$ and 
 $\omega= \omega_0^{\frac{b(1-\theta)}{p_0}} \omega_1^{\frac{b\theta}{p_1}}.$
\end{cor} 

Using the above corollary we now present the Hausdorff-Young-Paley inequality. 
\begin{thm}[Hausdorff-Young-Paley inequality] \label{HYP} Let $H$ be a commutative hypergroup equipped with a Haar measure $\lambda$ and let $\widehat{H}$ be the dual of $H$ equipped with measure $\pi.$ Let $1<p\leq 2,$ and let   $1<p \leq b \leq p' \leq \infty,$ where $p'= \frac{p}{p-1}.$ If $\psi(\chi)$ is a positive function on $\widehat{H}$ such that 
 \begin{equation}
	        M_\psi := \sup_{t>0} t \int_{\underset{\psi(\chi)\geq t}{\chi \in \widehat{H}}} d\pi(\chi)
	        \end{equation}
is finite, then for every $f \in L^p(H, d\lambda)$ 
 we have
\begin{equation} \label{Vish5.9}
    \left( \int_{\widehat{H}}  \left( |\widehat{f}(\chi)| \psi(\chi)^{\frac{1}{b}-\frac{1}{p'}} \right)^b d\pi(\chi)  \right)^{\frac{1}{b}} \lesssim M_\varphi^{\frac{1}{b}-\frac{1}{p'}} \|f\|_{L^p(H, d\lambda)}.
\end{equation}
\end{thm}
This naturally reduces to the Hausdorff-Young inequality when $b=p'$ and to the Paley inequality \eqref{Paleyin} when $b=p.$
\begin{proof}
From Theorem \ref{Paley}, the operator  defined by 
$$Af(\chi)= \widehat{f}(\chi),\,\,\,\,\chi \in \widehat{H}$$
is bounded from $L^p(H, d\lambda)$ to $L^{p}(\widehat{H},\omega_0  d\mu'),$ where $d\mu'(\chi)=d\pi(\chi)$ and $\omega_{0}(\chi)=  \psi(\chi)^{2-p}.$ From Theorem \ref{HY}, we deduce that $A:L^p(H, d\lambda) \rightarrow L^{p'}(\widehat{H}, \omega_1 d\mu')$ with $d\mu'(\chi)=d\pi(\chi)$ and   $\omega_1(\chi)= 1$  admits a bounded extension. By using the real interpolation (Corollary \ref{interpolationoperator} above) we will prove that $A:L^p(H, d\lambda) \rightarrow L^{b}(\widehat{H}, \omega d\mu'),$ $p\leq b\leq p',$ is bounded,
where the space $L^p(\widehat{H},\, \omega d\mu')$ is defined by the norm 
$$\|\sigma\|_{L^p(\widehat{H},\, \omega d\mu')}:=\left( \int_{\widehat{H}} |\sigma(\chi)|^p w(\chi) \,d\mu'(\chi) \right)^{\frac{1}{p}}= \left( \int_{\widehat{H} } |\sigma(\chi)|^p w(\chi) d\pi(\chi) \right)^{\frac{1}{p}}$$
 and $\omega(\chi)$ is positive function over $\widehat{H}$ to be determined. To compute $\omega,$ we can use Corollary \ref{interpolationoperator}, by fixing $\theta\in (0,1)$ such that $\frac{1}{b}=\frac{1-\theta}{p}+\frac{\theta}{p'}$. In this case $\theta=\frac{p-b}{b(p-2)},$ and 
 \begin{equation}
     \omega= \omega_0^{\frac{p(1-\theta)}{p_0}} \omega_1^{\frac{p\theta}{p_1}}= \psi(\chi)^{1-\frac{b}{p'}}.     
 \end{equation}
 Thus we finish the proof.
\end{proof}
\subsection{$L^p$-$L^q$  boundedness of Fourier multipliers on commutative  hypergroups } 

We begin with the definition of Fourier multipliers on commutative  hypergroups. For a function $h \in L^\infty(\widehat{H}, d\pi),$ define the operator $T_h$ as 
$$\widehat{T_hf}(\chi)=h(\chi) \widehat{f}(\chi),\,\,\,\, \chi \in \widehat{H},$$ for all $f$ belonging to a suitable function space on $\widehat{H}.$ The operator $T_h$ is called the Fourier multiplier on $H$ with symbol $h.$ It is clear that  $T_h$ is a bounded operator on  $L^2(H, d\lambda)$ by the Plancherel theorem. A Fourier multiplier commutes with the translation operators. In fact, the Fourier multipliers can be  characterized using translation operators \cite{SSK,Bloom, Kumar}.
	            
	    In the next result, we show that if the symbol $h$ of a Fourier multipliers $A$ defined on $C_c(H)$ satisfies certain H\"ormander's condition then $T_h$ can be extended as a bounded linear operator from $L^p(H, d\lambda)$ to $L^q(H, d\lambda)$ for the range $1<p \leq 2 \leq q <\infty.$

\begin{thm} \label{Jacobimult}  Let $1<p \leq 2 \leq q<\infty$ and let $H$ be a commutative  hypergroup. Suppose that $T$ is a  Fourier multiplier with symbol $h.$ Then we have 
$$\|T\|_{L^p(H,\, d\lambda) \rightarrow L^q(H,\, d\lambda)}\lesssim \sup_{s>0} s \left[ \int_{\{ \chi \in \widehat{H}: |h(\chi)|\geq s\}} d\pi(\chi) \right]^{\frac{1}{p}-\frac{1}{q}}.$$
   \end{thm}
\begin{proof} 
 Let us first assume that $p \leq q',$ where $\frac{1}{q}+\frac{1}{q'}=1.$ Since $q' \leq 2,$ dualising  the Hausdorff-Young inequality gives that 
 \begin{align*}
     \|Tf\|_{L^q(H,\, d\lambda)} \leq \|\widehat{Tf}\|_{L^{q'}(\widehat{H}, \, d\pi)} = \|h \widehat{f}\|_{L^{q'}(\widehat{H}, d\pi)}.
 \end{align*}
 
 The case $q' \leq (p')'=p$ can be reduced to the case $p \leq q'$ as follows. Using the duality of $L^p$-spaces we have $\|T\|_{L^p(H,\, d\lambda) \rightarrow L^q(H,\, d\lambda)}= \|T^*\|_{L^{q'}(H,\, d\lambda) \rightarrow L^{p'}(H,\, d\lambda)}.$ The symbol of adjoint operator $T^*$  is equal to $\bar{h}$ and obviously we have $|\bar{h}|= |h|$ (see Theorem 4.2 in \cite{ARN1}). 
 Now, we are in a position to apply Theorem \ref{HYP}. Set $\frac{1}{p}-\frac{1}{q}=\frac{1}{r}.$ Now, by applying  Theorem \ref{HYP} with $\psi= |h|^r$ with $b=q'$ we get 
 $$\|h \widehat{f}\|_{L^{q'}(H, d\lambda)} \lesssim \left(  \sup_{s>0} s \int_{\underset{|h(\chi)|^r > s}{\chi \in \widehat{H}}} d\pi(\chi)   \right)^{\frac{1}{r}} \|f\|_{L^p(H,\, d\lambda)} $$ for all $f \in L^p(H, d\lambda),$ in view of $\frac{1}{p}-\frac{1}{q}=\frac{1}{q'}-\frac{1}{p'}=\frac{1}{r.}$ Thus, for $1<p \leq 2 \leq q<\infty,$ we obtain 
 
 $$\|Tf\|_{L^q(H, d\lambda)} \lesssim \left(  \sup_{s>0} s \int_{\underset{|h(\chi)|^r > s}{\chi \in \widehat{H}}} d\pi(\chi)   \right)^{\frac{1}{r}} \|f\|_{L^p(H, d\lambda)}.$$
 Further, the proof follows from the following inequality: 
 \begin{align*}
     \left(  \sup_{s>0} s \int_{\underset{|h(\chi)|^r > s}{\chi \in \widehat{H}}} d\pi(\chi)  \right)^{\frac{1}{r}} &= \left(  \sup_{s>0} s \int_{\underset{|h(\chi)| > s^{\frac{1}{r}}}{\chi \in \widehat{H}}} d\pi(\chi)    \right)^{\frac{1}{r}} \\&=  \left(  \sup_{s>0} s^{\frac{1}{r}} \int_{\underset{|h(\chi)| > s} {\chi \in \widehat{H}}} d\pi(\chi)    \right)^{\frac{1}{r}} \\&= \sup_{s>0} s \left(   \int_{\underset{|h(\chi)| > s} {\chi \in \widehat{H}}} d\pi(\chi)    \right)^{\frac{1}{r}},
 \end{align*} proving Theorem \ref{Jacobimult}.
\end{proof}

\section{$L^p$-$L^q$  multipliers on Ch\'{e}bli-Trim\`{e}che hypergroups} \label{sec4}

\subsection{Interlude on Ch\'{e}bli-Trim\`{e}che hypergroups} \label{sec4.1}
An important and motivating example of a hypergroup is the algebra of finite radial measures on a noncompact rank one symmetric space under the usual convolution; as radial measures on a noncompact symmetric space can be viewed as measures on $\R_+:=[0,  \infty),$ this endows $\R_+$ with a hypergroup structure with involution being the identity map. The hypergroup structures arising this way on $\R_+$ is a particular case of general class of hypergroup structure on $\R_+$ arising from the Strum-Liouville boundary value problems where the solutions coincide with the characters of the hypergroup in question \cite[Section 3.5]{BX}. 
	
	In this paper, we are interested in a special class of ``one dimensional hypergroups'' \cite{Zeuner} on $\R_+$ called {\it Ch\'{e}bli-Trim\`{e}che hypergroups} with the convolution structure related to the following second-order differential operator 
	\begin{equation} \label{Visht}
	L=L_{A,x}:= -\frac{d^2}{dx^2}- \frac{A'(x)}{A(x)} \frac{d}{dx}, 
	\end{equation}
	where the function $A,$ called a {\it Ch\'{e}bli-Trim\`{e}che function}, is continuous on $\R_+,$ twice continuously  differentiable on $\R^*_+:=(0, \infty)$ and satisfies the following properties: 
	\begin{enumerate}
		\item[(i)] $A(0)=0$ and $A$ is positive on $\R^*_+.$
		\item[(ii)] $A$ is an increasing function and $A(x) \rightarrow \infty$ as $x \rightarrow \infty.$
		\item [(iii)] $\frac{A'}{A}$ is a decreasing $C^\infty$- function on $\R^*_+$ and hence $\rho:= \frac{1}{2} \lim\limits_{x \rightarrow \infty} \frac{A'(x)}{A(x)} \geq 0$ exists.  
		\item [(iv)] $\frac{A'(x)}{A(x)}= \frac{2\alpha}{x}+B(x)$ on a neighbourhood of $0$, where $\alpha>-\frac{1}{2}$ and $B$ is an odd $C^\infty$-function on $\R.$ 
	\end{enumerate}  
	
	Now, we define the main object of this section, namely, the Ch\'{e}bli-Trim\`{e}che hypergroups. For a detailed study of these hypergroups one can refer to \cite{Bloom,Kha297,Kha97,Kha81,Houcine,Houcine1,redbook}.
	\begin{defn} \cite{BX,BX1998,Zeuner}
		A hypergroup $(\R_+, *)$ is called a Ch\'{e}bli-Trim\`{e}che  hypergroup if there exists a Ch\'{e}bli-Trim\`{e}che function $A$ such that for  any real-valued $C^\infty$-function $f$ on $\R_+,$ i.e., the restriction of an even non-negative $C^\infty$-function on $\R,$ the generalised translation $u(x,y)=T_xf(y):=  \int_0^\infty f(t)\, d(\delta_x*\delta_y)(t),\,\,\,\, y \in \R_+$ is  the solution of the following Cauchy problem: 
		\begin{eqnarray*}
			&(L_{A,x}-L_{A,y})u(x, y)=0,\\& u(x,0)=f(x),\,\,\, u_y(x,0)= 0,\,\,\,\, x>0.
		\end{eqnarray*}
	\end{defn}
	The Ch\'{e}bli-Trim\`{e}che hypergroup associated with the Ch\'{e}bli-Trim\`{e}che function $A$ will be denoted by $(\R_+, *(A)).$ The growth of $(\R_+, *(A))$ is determined by the number $\rho:= \frac{1}{2} \lim\limits_{x \rightarrow \infty} \frac{A'(x)}{A(x)}.$ We say that $(\R_+, *(A))$ is of {\it exponential growth} if and only if $\rho >0.$ Otherwise we say that the hypergroup is of {\it subexponential growth} which also includes the polynomial growth. 
	
	
	The class of  Ch\'{e}bli-Trim\`{e}che hypergroups contains many important classes of hypergroups. We will discuss three of them here. 
	\begin{itemize}
		\item[(i)] If $A(x):= x^{2 \alpha+1}$ with $2 \alpha \in \mathbb{N}$ and $\alpha > -\frac{1}{2}$ then $L_{A,x}$ is the radial part of the Laplace operator on the Euclidean space and $(\R_+, *(A))$ is a Bessel-Kingman hypergroup. 
		\item [(ii)] If $A(x):= (sinh\, x)^{2\alpha+1} (cosh\,x)^{2 \beta+1}$ with  $2 \alpha, 2 \beta \in \mathbb{N},\,$ $\alpha \geq \beta \geq -\frac{1}{2}$ and $ \alpha \neq -\frac{1}{2}$ then $L_{A,x}$ is the radial part of the Laplace-Beltrami operator on a noncompact Riemannian symmetric space of rank one (also of Damek-Ricci spaces) and $(\R_+, *(A))$ is a Jacobi hypergroup (\cite{Anker96,Bloom}). 
		\item[(iii)] If $A$ is the density function on the simply connected harmonic manifold $X$ of purely exponential volume growth then $L_{A,x}$ is the radial part of the Laplace-Beltrami operator on $X$ and $(\R_+, *(A))$ is the ``radial hypergroup" of $X$ (see \cite{Biswas}).    
		
	\end{itemize}
	
	The Ch\'{e}bli-Trim\`{e}che hypergroup $(\R_+, *(A))$ is a noncompact commutative hypergroup with the identity element $0$ and involution as the identity map. The Haar measure $m$ on $(\R_+, *(A))$ is given by $m(x):= A(x)\,dx,$ where $dx$ is the usual Lebesgue measure on $\R_+.$ For any $x, y \in \R_+,$ the probability measure $\delta_x*\delta_y$ is absolutely continuous with respect to $m$ and 
	$\spt(\delta_x*\delta_y) \subset [\,|x-y|, x+y].$ For $1 \leq p \leq \infty,$ the Lebesgue space $L^p(\R_+, m)$ on $(\R_+, *(A))$ is defined as usual; we denote by $\|f\|_{p, A}$ the $L^p$-norm of $f \in L^p(\R_+, Adx).$ 
	
	For the Ch\'{e}bli-Trim\`{e}che hypergroup $(\R_+, *(A)),$ the multiplicative functions on $(\R_+, *(A))$ are given by the eigenfunctions of the operator $L:=L_{A,x}$ defined in \eqref{Visht}. For any $\lambda \in \mathbb{C},$ the equation 
	\begin{equation} \label{visht1}
	Lu = (\lambda^2+\rho^2)u
	\end{equation}
	has a unique solution $\phi_{\lambda}$ on $\R_+^*$ which extends continuously to $0$ and satisfies $\phi_{\lambda}(0)=1.$ We would like to point out here that since the coefficient $A'/A$ of $L$ is singular at $x =0$ as $A(0)=0,$ the existence of a solution continuous at $0$ is not immediate. Since \eqref{visht1} reads the same for $\lambda$ and $-\lambda,$ by uniqueness we have $\phi_{\lambda}= \phi_{-\lambda}.$
	The multiplicative functions, i.e., semicharacters, on $(\R_+, *(A))$  are then exactly the functions $\phi_{\lambda},\,\, \lambda \in \mathbb{C}.$ But a semicharacter  $\phi_{\lambda}$ is bounded, and therefore becomes a character, if and only if $|\text{Im} \lambda| \leq \rho.$ Since the involution of $(\R_+, *(A))$ is just the identity map it follows that the characters $\phi_\lambda$ of the hypergroup $(\R_+, *(A))$ are real-valued which is the case for every $\lambda \in \R \cup i \R.$ Therefore, in  view of above discussion, the dual space $\widehat{\R}_+$ of $(\R_+, *(A))$ is described by $\{\phi_{\lambda}: \lambda \in [0, \infty) \cup  [0, i \rho] \}$ which can be identified with the parameter set $\R_+ \cup [0, i \rho].$ 
	
	We define the Fourier transform $\widehat{f}$ of $f \in L^1(\R_+, Adx)$ at a point $\lambda \in \widehat{\R}_+$ by 
	$$\widehat{f}(\lambda):= \int_{0}^\infty f(x)\, \phi_{\lambda}(x)\, A(x)dx.$$
	It is worth noting that the above Fourier transform turns to be Jacobi transform \cite{Koorn} when the hypergroup arises from convolution of radial measures on a rank one symmetric space, or, equivalently, $A(x):= (sinh\, x)^{2\alpha+1} (cosh\,x)^{2 \beta+1}$ with $\alpha \geq \beta \geq -\frac{1}{2}$ and $ \alpha \neq -\frac{1}{2}.$
	The following theorem is the Levitan-Plancherel theorem for $(\R_+, *(A)).$
	\begin{thm} \cite[Theorem 2.1]{Bloom}
		There exists a unique non-negative measure $\pi$ on $\widehat{\R}_+$ with support $[0, \infty)$ such that the Fourier transform induces an isometric isomorphism from $L^2(\R_+, Adx)$ onto $L^2(\widehat{\R}_+, \pi)$ and for any $f \in L^1(\R_+, Adx) \cap L^2(\R_+, Adx)$, 
		$$\int_0^\infty |f(x)|^2 \, A(x)dx = \int_{\widehat{\R}_+}\,|\widehat{f}(\lambda)|^2\, d\pi(\lambda).$$
	\end{thm}
	
	In addition to the above theorem, we also have the following identity known as Parseval's identity. For $f_1, f_2 \in L^2(\mathbb{R}_+, Adx)$ we have 
	\begin{equation*}
	    \int_0^\infty f_1(x)\, f_1(x)\, A(x)dx = \int_{\widehat{\R}_+} \widehat{f}_1(\lambda) \, \widehat{f}_2(\lambda)\, d\pi(\lambda).
	\end{equation*}

	{\bf Condition (P):} We say that a function $f$ satisfies {\it Condition (P)} if for some $a>0,$ $f$ can be expressed as 
	$$ f(x)=\frac{a^2-\frac{1}{4}}{x^2}+\zeta(x)$$ for all large $x$, where 
	$$ \int_{x_0}^\infty x^{\gamma(a)} |\zeta(x)|\, dx<\infty$$ for some $x_0>0$ and $\zeta(x)$ is bounded for $x>x_0;$ here $\gamma(a)= a+\frac{1}{2}$ if $a \geq \frac{1}{2}$ and $\gamma(a)=1$ otherwise. 
	
	Now consider the function $G$ defined by 
	\begin{equation}
	G(x):= \frac{1}{4} \left( \frac{A'(x)}{A(x)} \right)^2+ \frac{1}{2} \frac{d}{dx} \left( \frac{A'(x)}{A(x)}\right) -\rho^2.
	\end{equation}
	It is possible to determine the Plancherel measure $\pi$ explicitly. In fact, Bloom and Xu \cite{BX} determined it by placing an extra growth condition  on $A$ described as follows:
	
	\begin{thm} \cite[Proposition 3.17]{BX} \label{bxthe}
		If the function $G$ defined above satisfies Condition (P) with any of the following conditions:
		\begin{itemize}
			\item [(i)] $a>\frac{1}{2};$ 
			\item [(ii)] $a \neq |\alpha|,$ where $\alpha$ is the constant appearing in the definition of Ch\'ebli-Trim\'eche  function. 
			\item [(iii)] $a=\alpha \leq \frac{1}{2}$ and $\int_0^\infty x^{\frac{1}{2} -\alpha} \zeta(x) \phi_0(x) A(x)^{\frac{1}{2}} dx \neq -2 \alpha \sqrt{M_A}$ or\\ $\int_0^\infty x^{\frac{1}{2} +\alpha} \zeta(x) \phi_0(x) A(x)^{\frac{1}{2}} dx=0$ where $M_A= \lim\limits_{x \rightarrow 0^+} x^{-2 \alpha -1} A(x)$ and\\ $\zeta(x)= G(x)+ (\frac{1}{4}-a^2)/x^2,$
		\end{itemize} then the Plancherel measure $\pi$ is absolutely continuous with respect to the Lebesgue measure and has density $|c(\lambda)|^2,$ where the function $c$ satisfies the following:  there exist positive constants $C_1, C_2$ and $K$ such that for any $\lambda \in \mathbb{C}$ with $\textnormal{Im}(\lambda) \leq 0,$
		$$C_1 |\lambda|^{a+\frac{1}{2}} \leq |c(\lambda)|^{-1} \leq C_2 |\lambda|^{a+\frac{1}{2}}\,\,\,\,\, \text{for}\quad |\lambda| \leq K,$$
		$$C_1 |\lambda|^{\alpha+\frac{1}{2}} \leq |c(\lambda)|^{-1} \leq C_2 |\lambda|^{\alpha+\frac{1}{2}}\,\,\,\,\, \text{for}\quad |\lambda| > K.$$
	\end{thm}
	In the sequel we always assume that the function $A$ (and so $G$) satisfies the Condition (P) of Theorem \ref{bxthe} for $\alpha \geq 0$. Therefore, $d\pi = C_0|c(\lambda)|^{-2} d\lambda,$ where $C_0$ is a positive constant and $c$ is a certain complex function on $\mathbb{C} \backslash \{0\}.$ We also assume that if case $\rho=0,$ $A$ also satisfies $A(x)=O(x^{2\alpha+1}),$ as $ x \rightarrow \infty.$
	
	
	Note that the space of infinitely differentiable and compactly supported functions on $\R_+,$ denoted by, $C_c^\infty(\R_+)$ is the space of all even $f \in C_c^\infty(\R)$  restricted to $\R_+.$
	
	We have the following Fourier inversion formula, for $f \in C_c^\infty(\R_+),$
	\begin{equation*}
	f(x)= C_0 \int_0^\infty \widehat{f}(\lambda) \phi_{\lambda}(x)\, |c(\lambda)|^{-2}\, d\lambda.
	\end{equation*}

For any $\lambda \in \mathbb{C},$ $\phi_\lambda$ is an even $C^\infty$-function and $\lambda \mapsto \phi_\lambda(x)$ is analytic. Also, $|\phi_{\lambda}(x)| \leq 1$ for all $x \in \R_+, \lambda \in \mathbb{C}$ with $|\textnormal{Im}\,\lambda| \leq \rho.$ For each $\lambda \in \mathbb{C},$ $\phi_{\lambda}$ has a Laplace representation $$\phi_{\lambda}(x)= \int_{-x}^{x} e^{(i \lambda-\rho)t} d\nu_x(t)\,\,\,\,\text{for}\,\, x \in \R_+,$$ 
	where $\nu_x$ is a probability measure on $\R$ supported in $[-x,x].$ Further, if $\lambda \in \mathbb{C}$ then $|\phi_{\lambda}| \leq C (1+x) e^{(Im \, \lambda-\rho)x}.$

%
	

\subsection{$L^p$-$L^q$ boundedness of multipliers on Ch\'{e}bli-Trim\`{e}che hypergroups}
The following theorem is an analogue of the Paley inequality for Ch\'{e}bli-Trim\`{e}che hypergroups on half line.

\begin{thm}
	      Suppose that $\psi$ is a positive function  on $\mathbb{R}_+$  satisfying the condition 
	    \begin{equation} \nonumber
	        M_\psi := \sup_{t>0} t \int_{\underset{\psi(\lambda)\geq t}{\lambda \in \mathbb{R}_+}} |c(\lambda)|^{-2}\, d\lambda <\infty.
	        \end{equation}
	        Then for  $f \in L^p(\mathbb{R}_+, Adx),$ $1<p\leq 2,$ we have 
	        \begin{align*}
	            \left( \int_0^\infty |\widehat{f}(\lambda)|^p\, \psi(\lambda)^{2-p} |c(\lambda)|^{-2}\, d\lambda \right)^{\frac{1}{p}} \lesssim M_{\psi}^{\frac{2-p}{p}}\, \|f\|_{L^p(\mathbb{R}_+, Adx)}.
	        \end{align*}
	\end{thm}
The following Hausdorff-Young-Paley inequality holds for the Ch\'{e}bli-Trim\`{e}che hypergroups. 
\begin{thm}  Let $1<p\leq 2,$ and let   $1<p \leq b \leq p' \leq \infty,$ where $p'= \frac{p}{p-1}.$ If $\psi(\lambda)$ is a positive function on $\mathbb{R}_+$ such that 
 \begin{equation}
	        M_\psi := \sup_{t>0} t \int_{\underset{\psi(\lambda)\geq t}{\lambda \in \mathbb{R}_+}} |c(\lambda)|^{-2}\, d\lambda
	        \end{equation}
is finite, then for every $f \in L^p(\mathbb{R}_+, Adx)$ 
 we have
\begin{equation} 
    \left( \int_{\mathbb{R}_+}  \left( |\widehat{f}(\lambda)| \psi(\lambda)^{\frac{1}{b}-\frac{1}{p'}} \right)^b |c(\lambda)|^{-2}\, d\lambda  \right)^{\frac{1}{b}} \lesssim M_\varphi^{\frac{1}{b}-\frac{1}{p'}} \|f\|_{L^p(\mathbb{R}_+, Adx)}.
\end{equation}
\end{thm}
Next, we establish the following $L^p$-$L^q$ boundedness result for Fourier multipliers on Ch\'{e}bli-Trim\`{e}che hypergroups.
\begin{thm} \label{pqF} Let $1<p \leq 2 \leq q<\infty$ and suppose that $T$ is a  Fourier multiplier with symbol $h.$ Then we have 
$$\|T\|_{L^p(\mathbb{R}_+, Adx) \rightarrow L^q(\mathbb{R}_+, Adx)}\lesssim \sup_{s>0} s \left[ \int_{\{ \lambda \in \mathbb{R}_+: |h(\lambda)|\geq s\}} |c(\lambda)|^{-2}\, d\lambda \right]^{\frac{1}{p}-\frac{1}{q}}.$$
   \end{thm}

\begin{rem}
It is well-known that it is necessary for a Fourier multiplier $T_h$ to be  bounded from $L^p(\mathbb{R}_+, Adx)$ to $L^p(\mathbb{R}_+, Adx),$ that the symbol $h$ must be holomorphic in a strip in the complex plane (see \cite{Anker96, BX}). It is evident from our theorem, it is no longer a necessary condition for $T_h$ to bounded from $L^p$ to $L^q,$ for $1<p \leq 2 \leq q<\infty.$ 
\end{rem}



\subsection{Spectral multipliers of the generalised radial Laplacian}
Now, we apply Theorem \ref{Jacobimult} to prove the $L^p$-$L^q$ boundedness of spectral multipliers for operator $L:=L_{A, x}.$ If $\varphi \in L^\infty(\R_+, Adx),$ the spectral multiplier $\varphi(L),$ defined by $\varphi$ coincides with the Ch\'{e}bli-Trim\`{e}che Fourier multiplier $T_h$ with $h(\lambda)= \varphi(\lambda^2+\rho^2)$ for $\lambda \in \R_+.$   The $L^p$-boundedness of spectral multipliers  has been proved by several authors in many different setting, e.g., Bessel transform \cite{BCC}, Dunkl harmonic oscillator \cite{Wro}, multidimesional Hankel transform \cite{DPW}. The $L^p$-$L^q$ boundedness of spectral multipliers for compact Lie groups, Heisenberg groups, graded Lie groups have been proved by R. Akylzhanov and the second author \cite{AR}. M. Chatzakou and the first author recently studied $L^p$-$L^q$ boundedness of spectral multiplier for the anharmonic oscillator \cite{CK} (see also \cite{CKNR}). Very recently, the authors \cite{KR21}  studied spectral multipliers of Jacobi Laplacian.

\begin{thm} \label{specmul}
Let $1<p \leq 2 \leq q <\infty$ and let $\varphi$ be a monotonically  decreasing continuous function on $[\rho^2, \infty)$ such that $\lim_{u \rightarrow \infty}\varphi(u)=0.$ Then we have 
\begin{equation}
    \|\varphi(L)\|_{\textnormal{op}} \lesssim \sup_{u>\rho^2} \varphi(u)  \begin{cases} (u-\rho^2)^{(a+1)(\frac{1}{p}-\frac{1}{q})} & \quad \textnormal{if} \quad (u-\rho^2)^{\frac{1}{2}} \leq K, \\ \left[K^{2a+2}- K^{2\alpha+2}+(u-\rho^2)^{(\alpha+1)} \right]^{\frac{1}{p}-\frac{1}{q}} & \quad \textnormal{if} \quad (u-\rho^2)^{\frac{1}{2}}>K,  \end{cases}
\end{equation} where $K$ is a constant appearing in the estimate of the  $c$-function  and $\|\cdot\|_{\textnormal{op}} $ denotes the operator norm from $L^p(\mathbb{R}_+, A dx)$ to $L^q(\mathbb{R}_+, Adx).$ 
\end{thm}
\begin{proof} Since $\varphi(L)$ is a Fourier multiplier with the symbol $\varphi(\lambda^2+\rho^2),$ as an application of Theorem \ref{Jacobimult}, we get
\begin{align*}
    \|\varphi(L)\|_{\textnormal{op}} & \lesssim \sup_{s>0} s \left[ \int_{ \{ \lambda \in \mathbb{R}_+ :\, \varphi(\lambda^2+\rho^2) \geq s\} } |c(\lambda)|^{-2}\,d\lambda \right]^{\frac{1}{p}-\frac{1}{q}} \\ & = \sup_{0<s<\varphi(\rho^2)} s \left[ \int_{ \{ \lambda \in \mathbb{R}_+ :\, \varphi(\lambda^2+\rho^2) \geq s\} } |c(\lambda)|^{-2}\,d\lambda \right]^{\frac{1}{p}-\frac{1}{q}}, 
\end{align*} since $\varphi \leq \varphi(\rho^2).$ Now, as $s \in (0, \varphi(\rho^2)]$ we can write $s= \varphi(u)$ for some $u \in [\rho^2, \infty)$ and, therefore, we have 
\begin{align*}
    \|\varphi(L)\|_{\textnormal{op}} \lesssim \sup_{\varphi(u) <\varphi(\rho^2)} \varphi(u)  \left[ \int_{ \{ \lambda \in \mathbb{R}_+ :\, \varphi(\lambda^2+\rho^2) \geq \varphi(u)\} } |c(\lambda)|^{-2}\,d\lambda \right]^{\frac{1}{p}-\frac{1}{q}}.
\end{align*} Since $\varphi$ is monotonically decreasing we get 
\begin{align*}
    \|\varphi(L)\|_{\textnormal{op}} \lesssim \sup_{u >\rho^2} \varphi(u) \left[ \int_{ \{ \lambda \in \mathbb{R}_+ :\, \lambda \leq (u-\rho^2)^{\frac{1}{2}}\} } |c(\lambda)|^{-2}\,d\lambda \right]^{\frac{1}{p}-\frac{1}{q}}.
\end{align*}
Now, we use the estimate of the $c$-function (see Theorem \ref{bxthe}) to get 
\begin{align*}
     \|\varphi(L)\|_{\textnormal{op}} &\lesssim  \sup_{u >\rho^2} \varphi(u) \begin{cases}  \left[ \int_0^{ (u-\rho^2)^{\frac{1}{2}} } \lambda^{2a+1} \,d\lambda \right]^{\frac{1}{p}-\frac{1}{q}} \quad &\textnormal{if} \quad   (u-\rho^2)^{\frac{1}{2}} \leq K, \\ \left[ \int_0^K \lambda^{2a+1} d\lambda +\int_K^{ (u-\rho^2)^{\frac{1}{2}} } \lambda^{2\alpha+1}\,d\lambda \right]^{\frac{1}{p}-\frac{1}{q}} \quad &\textnormal{if} \quad  (u-\rho^2)^{\frac{1}{2}}>K
     \end{cases} \\&=\sup_{u>\rho^2} \varphi(u)  \begin{cases} (u-\rho^2)^{(a+1)(\frac{1}{p}-\frac{1}{q})} & \quad \textnormal{if} \quad (u-\rho^2)^{\frac{1}{2}} \leq K, \\ \left[K^{2a+2}- K^{2\alpha+2}+(u-\rho^2)^{(\alpha+1)} \right]^{\frac{1}{p}-\frac{1}{q}} & \quad \textnormal{if} \quad (u-\rho^2)^{\frac{1}{2}}>K,  \end{cases}
\end{align*} proving Theorem \ref{specmul}.
\end{proof}

\subsubsection{Heat equation}
Let us consider the $L$-heat equation 
\begin{equation} \label{23vis}
    \partial_t u+L u=0, \quad u(0)=u_0.
\end{equation} One can verify that for each $t>0,$ $u(t, x)=e^{-t L}u_0$ is a solution of initial value problem \eqref{23vis}. To apply Theorem \ref{specmul} we consider the function $\varphi(u)= e^{-tu}$ which satisfies the condition of Theorem \ref{specmul} and therefore we get 
\begin{equation}
    \|e^{-tL}\|_{\textnormal{op}} \leq \sup_{u>\rho^2} e^{-tu} \begin{cases} (u-\rho^2)^{(a+1)(\frac{1}{p}-\frac{1}{q})} & \quad \textnormal{if} \quad (u-\rho^2)^{\frac{1}{2}} \leq K, \\ \left[K^{2a+2}- K^{2\alpha+2}+(u-\rho^2)^{(\alpha+1)} \right]^{\frac{1}{p}-\frac{1}{q}} & \quad \textnormal{if} \quad (u-\rho^2)^{\frac{1}{2}}>K,  \end{cases}
    \end{equation}
    by setting $s= (u-\rho^2)^{\frac{1}{2}}$ we get 
    \begin{align} \label{eq28}
        \|e^{-tL}\|_{\textnormal{op}} \leq e^{-t \rho^2}\sup_{s>0} e^{-ts^2} \begin{cases} s^{2(a+1)(\frac{1}{p}-\frac{1}{q})} & \quad \textnormal{if} \quad s \leq K, \\ \left[K^{2a+2}- K^{2\alpha+2}+s^{2(\alpha+1)} \right]^{\frac{1}{p}-\frac{1}{q}} & \quad \textnormal{if} \quad s >K,  \end{cases}
    \end{align}
    Now, we will first calculate  $e^{-t \rho^2}\,\,\sup_{s>0} e^{-ts^2} s^{\frac{2(a+1)}{r}}$ where $\frac{1}{r}=\frac{1}{p}-\frac{1}{q}.$ Consider the function $$\phi(s)= e^{-ts^2} s^{\frac{2(a+1)}{r}}.$$
    Then $$\phi'(s)= e^{-ts^2}\left(-2st+\frac{2(a+1)}{r} \right) s^{\frac{2(a+1)}{4}-1}.$$
    So, $\phi$ will be zero only at $s_0= \frac{(a+1)}{rt}$ and derivative changes the sign from positive to negative at $s_0.$ Thus, $s_0$ is point of maximum of $\phi$ if $s_0 \leq K$, i.e., $\frac{(a+1)}{rt} \leq K$ implies that $t \geq \frac{(a+1)}{r K}.$  Thus, 
    \begin{align} \label{eq29}
        e^{-t \rho^2}\,\,\sup_{s>0} e^{-ts^2} s^{\frac{2(a+1)}{r}}= e^{-\frac{(a+1)^2}{ r^2t}} \left(\frac{a+1}{rt} \right)^{\frac{2(a+1)}{r}} \leq  C_{a,r}\,\, e^{-t \rho^2}\,\, e^{-\frac{(a+1)^2}{ r^2t}} t^{\frac{-2(a+1)}{r}}.
    \end{align}
    Now, for $s >K,$ consider the function 
    $\psi(s)= e^{-ts^2} [K^{2a+2}-K^{2\alpha+2}+s^{2(\alpha+1)}]^{\frac{1}{r}}.$ Then $\psi$ attains its maximum at $s_0= \frac{(\alpha+1)}{r t}$ if $s_0>K,$ i.e., $\frac{(\alpha+1)}{r t}>K$ implies that $0<t<\frac{\alpha+1}{rK}.$ So, 
    \begin{align} \label{eq30}
        e^{-t \rho^2}\,\,\sup_{s>0}\, e^{-ts^2} [K^{2a+2}-K^{2\alpha+2}+s^{2(\alpha+1)}]^{\frac{1}{r}}&=e^{-t \rho^2}\,\, e^{- \frac{(\alpha+1)}{r^2 t}} \left[K^{2a+2}-K^{2\alpha+2}+ \left( \frac{(\alpha+1)}{rt}\right)^{2(\alpha+1)} \right]^{\frac{1}{r}} \nonumber \\& \leq e^{-\frac{K}{r}}\left[K^{2a+2}-K^{2\alpha+2}+ \left( \frac{(\alpha+1)}{rt}\right)^{2(\alpha+1)} \right]^{\frac{1}{r}} \nonumber \\&\leq C_{\alpha, r, K}\, t^{-\frac{2(\alpha+1)}{r}}.   
    \end{align}
    Next, by putting the estimates from \eqref{eq29} and \eqref{eq30} in \eqref{eq28} we obtain
    \begin{align} \label{eq31}
        \|e^{-tL}\|_{\textnormal{op}} \lesssim \begin{cases} t^{-2(\alpha+1)(\frac{1}{p}-\frac{1}{q})} \quad & \text{if}\quad 0<t < \frac{\alpha+1}{K}\left(\frac{1}{p}-\frac{1}{q} \right) \\e^{-t \rho^2}\,\, e^{-\frac{(a+1)^2}{t}\left(\frac{1}{p}-\frac{1}{q}\right)^2} t^{-2(a+1)(\frac{1}{p}-\frac{1}{q})} \quad & \text{if}\quad t \geq  \frac{a+1}{K}\left(\frac{1}{p}-\frac{1}{q} \right).\end{cases}
    \end{align}
    Therefore, 
    \begin{align}
    \|u(t, \cdot)\|_{L^q(\mathbb{R}_+, Adx)} &\leq \|e^{-t L}u_0\|_{L^q(\mathbb{R}_+, Adx)} \leq \|e^{-tL}\|_{\text{op}} \|u\|_{L^p(\mathbb{R}_+, Adx)}\nonumber \\&\lesssim \|u\|_{L^p(\mathbb{R}_+, Adx)}\begin{cases} t^{-2(\alpha+1)(\frac{1}{p}-\frac{1}{q})}  & \text{if}\quad 0<t < \frac{\alpha+1}{K}\left(\frac{1}{p}-\frac{1}{q} \right) \\e^{-t \rho^2}\,\, e^{-\frac{(a+1)^2}{t}\left(\frac{1}{p}-\frac{1}{q}\right)^2} t^{-2(a+1)(\frac{1}{p}-\frac{1}{q})}  & \text{if}\quad t \geq  \frac{a+1}{K}\left(\frac{1}{p}-\frac{1}{q} \right).
    \end{cases}
    \end{align}

\subsubsection{Sobolev type embedding} Let $1<p\leq 2 \leq q<\infty.$ As an application of Theorem \ref{specmul}, we will obtain Sobolev type embedding theorem for operator $L.$  Consider the function $\varphi(u)= (1+u)^{-b},$ $u > \rho^2,$  then $\varphi$ satisfies the condition of Theorem 
\ref{specmul}. Therefore, from Theorem \ref{specmul} we have 
$$\|(1+L)^{-b}\|_{\text{op}} \lesssim \sup_{u >\rho^2} (1+u)^{-b} \begin{cases} (u-\rho^2)^{(a+1)(\frac{1}{p}-\frac{1}{q})} & \quad \textnormal{if} \quad (u-\rho^2)^{\frac{1}{2}} \leq K, \\ \left[K^{2a+2}- K^{2\alpha+2}+(u-\rho^2)^{(\alpha+1)} \right]^{\frac{1}{p}-\frac{1}{q}} & \quad \textnormal{if} \quad (u-\rho^2)^{\frac{1}{2}}>K,  \end{cases} $$ which is finite
provided that $b \geq (\alpha+1) \left( \frac{1}{p}-\frac{1}{q} \right).$

Therefore, we have 
$$\|f\|_{L^q(\R_+, A dx)} \leq C \|(1+L)^b f\|_{L^q(\R_+, A dx)}$$ provided that $b \geq (\alpha+1) \left( \frac{1}{p}-\frac{1}{q} \right),\,\,1<p\leq 2 \leq q<\infty.$

\section{Applications to nonlinear PDEs}
 In this section, we present applications of our main result on $L^p$-$L^q$ boundedness of Fourier multipliers on Ch\'ebli-Trim\'eche hypergroups to the well-posedness of abstract Cauchy problems on positive real line. Here, we adopt the techniques developed in \cite{CKNR} in the case of Fourier analysis associated to the biorthogonal eigenfunction expansion of a model operator having discrete spectrum.  
 \subsection{Nonlinear heat equation}
 Let us consider the following Cauchy problem for a nonlinear evolution equation in the space $L^\infty(0, T, L^2(\mathbb{R}_+, Adx)),$
 \begin{equation} \label{heat}
     u_t-|Bu(t)|^p=0,\quad u(0)=u_0,
 \end{equation}
 where $B$ is a linear operator on $L^2(\mathbb{R}_+, Adx)$ and $1<p<\infty.$
 
 We say that the heat equation \eqref{heat} admits a solution $u$ if we have  
 \begin{equation}\label{heatsol}
     u(t)=u_0+\int_0^t |Bu(\tau)|^p\, d\tau
 \end{equation} in the space $L^\infty(0, T, L^p(\mathbb{R}_+, Adx))$ for every $T<\infty.$
 We say that $u$ is a local solution of \eqref{heat} if it satisfies  equation \eqref{heatsol} in the space $L^\infty(0, T^*, L^2(\mathbb{R}_+, Adx))$ for some $T^*>0.$
 \begin{thm}
 Let $1<p<\infty.$ Suppose that $B$ is a Fourier multiplier from $L^2(\mathbb{R}_+, Adx)$ into $L^{2p}(\mathbb{R}_+, Adx)$  such that its symbol $h$ satisfies $$\sup_{s>0} s \left[ \int_{\{ \lambda \in \mathbb{R}_+: |h(\lambda)|\geq s\}} |c(\lambda)|^{-2}\, d\lambda \right]<\infty.$$ Then the Cauchy problem \eqref{heat} has a local solution in the space $L^\infty(0, T^*, L^2(\mathbb{R}_+, Adx))$ for some $T^*>0.$ 
 \end{thm}
 \begin{proof}
 By integrating equation \eqref{heat} w.r.t. $t$ one get
$$
u(t)=u_{0} + \int\limits_0^t |Bu(\tau)|^{p} d\tau.
$$
By taking the $L^2$-norm on both sides, one obtains
\begin{align*}
    \|u(t)\|_{L^{2}(\mathbb{R}_+, Adx)}^{2}  &\leq C \Bigg(\|u_0\|_{L^2(\mathbb{R}_+, Adx)}^2+\left\| \int_0^t |Bu(t)|^p\, d\tau \right\|^2_{L^2(\mathbb{R}_+, Adx)} \Bigg)\\& = C \Bigg(\|u_0\|_{L^2(\mathbb{R}_+, Adx)}^2+ \int_{\mathbb{R}_+} \left| \int_0^t |Bu(t)|^p\, d\tau  \right|^2 A(x)\,dx\Bigg).
\end{align*}
Using the inequality $\int_0^t|Bu(\tau)|^p\, d\tau \leq (\int_0^t 1 \, d\tau)^{\frac{1}{2}} (\int_0^t|Bu(\tau)|^{2p}\, d\tau)^{\frac{1}{2}}= t^{\frac{1}{2}} (\int_0^t|Bu(\tau)|^{2p}\, d\tau)^{\frac{1}{2}}$ we get
\begin{align*}
    \|u(t)\|_{L^{2}(\mathbb{R}_+, Adx)}^{2}  &\leq C \Bigg(\|u_0\|_{L^2(\mathbb{R}_+, Adx)}^2+ t \int_{\mathbb{R}_+}   \int_0^t |Bu(t)|^{2p}\, d\tau\,  A(x)\,dx\Bigg)\\& \leq  C \Bigg(\|u_0\|_{L^2(\mathbb{R}_+, A\,dx)}^2+ t    \int_0^t \int_{\mathbb{R}_+} |Bu(t)|^{2p}\,   A(x)\,dx\, d\tau\Bigg) \\& \leq  C \Bigg(\|u_0\|_{L^2(\mathbb{R}_+, Adx)}^2+ t    \int_0^t \|Bu(t)\|^{2p}_{L^{2p}(\mathbb{R}_+, Adx)}\, d\tau\Bigg).
\end{align*}

Next, using the condition on the symbol it can be seen, as an application of Theorem \ref{pqF}, that the operator $B$ is a bounded operator, that is, $\|B u(t)\|_{L^{2p}(\mathbb{R}_+, A dx)} \leq C_1 \|u(t)\|_{L^{2}(\mathbb{R}_+, A dx)}$ and, therefore,  the above inequality yields
\begin{align}\label{EQ:space-norm}
    \|u(t)\|_{L^{2}(\mathbb{R}_+, A\,dx)}^{2} \leq C \Bigg(\|u_0\|_{L^2(\mathbb{R}_+, A dx)}^2+ t    \int_0^t  \|u(t)\|^{2p}_{L^{2}(\mathbb{R}_+, A dx)}\, d\tau\Bigg),
\end{align}
for some constant $C$ independent from $u_0$ and $t$.

Finally, by taking $L^{\infty}$-norm in time on both sides of the estimate \eqref{EQ:space-norm}, one obtains
\begin{equation}\label{EQ:time-space-norm}
\|u(t)\|_{L^{\infty}(0, T; L^{2}(\mathbb{R}_+, A dx))}^{2}\leq C\Big(\|u_{0}\|_{L^{2}(\mathbb{R}_+, A dx)}^{2} + T^{2} \|u\|^{2p}_{L^{\infty}(0, T; L^{2}(\mathbb{R}_+, A dx))}\Big).
\end{equation}

Let us introduce the following set
\begin{equation}\label{u-Set}
S_c:=\left\{u\in L^{\infty}(0, T; L^{2}(\mathbb{R}_+, A dx)): \|u\|_{L^{\infty}(0, T; L^{2}(\mathbb{R}_+, A dx))} \leq c \|u_{0}\|_{L^{2}(\mathbb{R}_+, A dx)}\right\},
\end{equation}
for some constant $c \geq 1$. Then, for $u \in S_c$ we have
$$
\|u_{0}\|_{L^{2}(\mathbb{R}_+, A dx)}^{2} + T^{2} \|u\|^{2p}_{L^{\infty}(0, T; L^{2}(\mathbb{R}_+, A dx))} \leq 
\|u_{0}\|_{L^{2}(\mathbb{R}_+, A dx)}^{2} + T^{2} c^{2p} \|u_0\|^{2p}_{L^{2}(\mathbb{R}_+, A dx)}.
$$
Finally, for $u$ to be from the set $S_c$ it is enough to have, by invoking \eqref{EQ:time-space-norm}, that 
$$
\|u_{0}\|_{L^{2}(\mathbb{R}_+, A dx)}^{2} + T^{2} c^{2p} \|u_0\|^{2p}_{L^{2}(\mathbb{R}_+, A dx)}\leq c^{2} \|u_0\|_{L^{2}(\mathbb{R}_+, A dx)}^{2}.
$$
It can be obtained by requiring the following,
$$
T \leq T^{\ast}:=\frac{\sqrt{c^{2}-1}}{c^{p}\|u_0\|_{L^{2}(\mathbb{R}_+, A dx)}}.
$$
Thus, by applying the fixed point theorem, there exists a unique local solution $u\in L^{\infty}(0, T^{\ast}; L^{2}(\mathbb{R}_+, A dx))$ of the Cauchy problem \eqref{heat}.
 \end{proof}
 
 \subsection{Nonlinear wave equation} In this subsection, we will consider that the initial value problem  for the nonlinear wave equation  
\begin{align}\label{E-WNLE}
u_{tt}(t) - b(t)|Bu(t)|^{p} = 0,
\end{align} with the initial condition 
$$
u(0)=u_0, \,\,\, u_t(0)=u_1,
$$
where $b$ is a positive bounded function depending only on time, $B$ is a linear operator in $L^2(\mathbb{R}_+, A dx)$ and $1< p<\infty$. We intend to study the well-posedness of the equation \eqref{E-WNLE}.

We say that initial valued problem \eqref{E-WNLE} admits a global solution $u$ if it satisfies
\begin{equation}\label{E-WNLE-Sol}
u(t)=u_{0} + t u_{1} + \int\limits_0^t (t-\tau) b(\tau) |Bu(\tau)|^{p} d\tau
\end{equation}
in the space $L^{\infty}(0, T; L^{2}(\mathbb{R}_+, A dx))$ for every $T<\infty$.

We say that \eqref{E-WNLE} admits a local solution $u$ if it satisfies
the equation \eqref{E-WNLE-Sol} in the space $L^{\infty}(0, T^{\ast}; L^{2}(\mathbb{R}_+, A dx))$ for some $T^{\ast}>0$.

\begin{thm}\label{Th: E-WNLE}
Let $1\leq p<\infty$. Suppose that $B$ is a Fourier multiplier from $L^2(\mathbb{R}_+, A dx)$ into $L^{2p}(\mathbb{R}_+, A dx)$  such that its symbol $h$ satisfies $$\sup_{s>0} s \left[ \int_{\{ \lambda \in \mathbb{R}_+: |h(\lambda)|\geq s\}} |c(\lambda)|^{-2}\, d\lambda \right]<\infty.$$

\begin{itemize}
    \item [(i)] If $\|b\|_{L^{2}(0, T)}<\infty$ for some $T>0$ then the Cauchy problem \eqref{E-WNLE} has a local solution in  $L^{\infty}(0, T; L^{2}(\mathbb{R}_+, A dx))$.
    \item [(ii)] Suppose that $u_1$ is identically equal to zero. Let $\gamma>3/2$. Moreover, assume that $\|b\|_{L^{2}(0, T)}\leq c \, T^{-\gamma}$ for every $T>0$, where $c$ does not depend on $T$. Then, for every $T>0$, the Cauchy problem \eqref{E-WNLE} has a global solution in the space $L^{\infty}(0, T; L^{2}(\mathbb{R}_+, A dx))$ for sufficiently small $u_0$ in $L^2$-norm.
\end{itemize}
\end{thm}
\begin{proof} (i) By integrating the equation \eqref{E-WNLE} two times  in $t$ one gets
$$
u(t)=u_{0} + t u_{1} + \int\limits_0^t (t-\tau) b(\tau) |Bu(\tau)|^{p} d\tau.
$$
By taking the $L^2$-norm on both sides, for $t<T$ one obtains by simple calculation that

\begin{align*}
    &\|u(t)\|_{L^{2}(\mathbb{R}_+, A dx)}^{2}\leq  C\left\{ \|u_{0}\|_{L^{2}(\mathbb{R}_+, A dx)}^{2} + t^2 \|u_{1}\|_{L^{2}(\mathbb{R}_+, A dx)}^{2}+ \left\|\int\limits_0^t (t-\tau) b(\tau) |Bu(\tau)|^{p} d\tau \right\|_{L^{2}(\mathbb{R}_+, A dx)}^{2} \right\} \\& \leq C\left\{ \|u_{0}\|_{L^{2}(\mathbb{R}_+, A dx)}^{2} + t^2 \|u_{1}\|_{L^{2}(\mathbb{R}_+, A dx)}^{2}+ \int_{\mathbb{R}_+} \Big|\int\limits_0^t (t-\tau) b(\tau) |Bu(\tau)|^{p} d\tau \Big|^2 A(x) dx \right\} \\& \leq C\left\{ \|u_{0}\|_{L^{2}(\mathbb{R}_+, A dx)}^{2} + t^2 \|u_{1}\|_{L^{2}(\mathbb{R}_+, A dx)}^{2}+ \int_{\mathbb{R}_+} \Big(t \int\limits_0^t \Big| b(\tau)|Bu(\tau)|^{p}\Big| d\tau \Big)^{2} A(x) dx \right\} \\& \leq C\left\{ \|u_{0}\|_{L^{2}(\mathbb{R}_+, A dx)}^{2} + t^2 \|u_{1}\|_{L^{2}(\mathbb{R}_+, A dx)}^{2}+ \int_{\mathbb{R}_+} t^{2} \int\limits_0^t \Big| b(\tau)\Big|^{2} d\tau \int\limits_0^t \Big| Bu(\tau)\Big|^{2p} d\tau A(x) dx \right\} \\& \leq C\left\{ \|u_{0}\|_{L^{2}(\mathbb{R}_+, A dx)}^{2} + t^2 \|u_{1}\|_{L^{2}(\mathbb{R}_+, A dx)}^{2}+ t^2 \|b\|_{L^2(0,T)}^2 \int_{\mathbb{R}_+}  \int\limits_0^t \Big| Bu(\tau)\Big|^{2p} d\tau A(x) dx \right\} \\& \leq C\left\{ \|u_{0}\|_{L^{2}(\mathbb{R}_+, A dx)}^{2} + t^2 \|u_{1}\|_{L^{2}(\mathbb{R}_+, A dx)}^{2}+ t^2 \|b\|_{L^2(0,T)}^2 \int\limits_0^t \|Bu(\tau)\|^{2p}_{L^{2p}(\mathbb{R}_+, A dx)} d\tau \right\}.
\end{align*}
Next, using the condition on the symbol it can be seen, as an application of Theorem \ref{pqF}, that the operator $B$ is a bounded operator, that is, $\|B u(t)\|_{L^{2p}(\mathbb{R}_+, A dx)} \leq C_1 \|u(t)\|_{L^{2}(\mathbb{R}_+, A dx)}$ and, therefore, the above inequality yields
\begin{equation}
\label{EQ: WE-space-norm}
\|u(t)\|_{L^{2}(\mathbb{R}_+, A dx)}^{2}\leq C(\|u_{0}\|_{L^{2}(\mathbb{R}_+, A dx)}^{2} + t^2 \|u_{1}\|_{L^{2}(\mathbb{R}_+, A dx)}^{2}+ t^{2} \|b\|_{L^{2}(0, T)}^{2} \int\limits_0^t \|u(\tau)\|^{2p}_{L^{2p}(\mathbb{R}_+, A dx)} d\tau),
\end{equation}
for some constant $C$ not depending on $u_0, u_1$ and $t$. Finally, by taking the $L^{\infty}$-norm in time on both sides of the estimate \eqref{EQ: WE-space-norm}, one obtains
\begin{equation}
\label{EQ: WE-time-space-norm}
\|u\|_{L^{\infty}(0, T; L^{2}(\mathbb{R}_+, A dx))}^{2}\leq C (\|u_{0}\|_{L^{2}(\mathbb{R}_+, A dx)}^{2} + T^2 \|u_{1}\|_{L^{2}(\mathbb{R}_+, A dx)}^{2}+ T^{3} \|b\|_{L^{2}(0, T)}^{2} \|u\|^{2p}_{L^{\infty}(0, T; L^{2}(\mathbb{R}_+, A dx))}). 
\end{equation}

Let us introduce the set
\begin{equation}
Q_c:=\Big\{u\in L^{\infty}(0, T; L^{2}(\mathbb{R}_+, A dx)): \|u\|_{L^{\infty}(0, T; L^{2}(\mathbb{R}_+, A dx))}^{2} \leq
c(\|u_{0}\|_{L^{2}(\mathbb{R}_+, A dx)}^{2} + T^2 \|u_{1}\|_{L^{2}(\mathbb{R}_+, A dx)}^{2})\Big\}
\end{equation}
for some constant $c \geq 1$. Then, for $u \in  Q_c$ we have
\begin{align}\label{WE-Est}
&\|u_{0}\|_{L^{2}(\mathbb{R}_+, A dx)}^{2} + T^2 \|u_{1}\|_{L^{2}(\mathbb{R}_+, A dx)}^{2} + T^{3} \|b\|_{L^{2}(0, T)}^{2} \|u\|^{2p}_{L^{\infty}(0, T; L^{2}(\mathbb{R}_+, A dx))} \nonumber \\
&\leq \|u_{0}\|_{L^{2}(\mathbb{R}_+, A dx)}^{2} + T^2 \|u_{1}\|_{L^{2}(\mathbb{R}_+, Adx)}^{2} + T^{3} \|b\|_{L^{2}(0, T)}^{2} c^{p}\Big(\|u_{0}\|_{L^{2}(\mathbb{R}_+, A dx)}^{2} + T^2 \|u_{1}\|_{L^{2}(\mathbb{R}_+, A dx)}^{2}\Big)^{p}. 
\end{align}

Observe that, to be $u$ from the set $Q_c$ it is enough to have, by invoking \eqref{EQ: WE-time-space-norm} and using \eqref{WE-Est}, that 
\begin{align*}
\begin{split}
\|u_{0}\|_{L^{2}(\mathbb{R}_+, A dx)}^{2} + &T^2 \|u_{1}\|_{L^{2}(\mathbb{R}_+, A dx)}^{2} + T^{3} \|b\|_{L^{2}(0, T)}^{2} c^{p}\Big(\|u_{0}\|_{L^{2}(\mathbb{R}_+, A dx)}^{2} + T^2 \|u_{1}\|_{L^{2}(\mathbb{R}_+, A dx)}^{2}\Big)^{p}\\
&\leq c(\|u_{0}\|_{L^{2}(\mathbb{R}_+, A dx)}^{2} + T^2 \|u_{1}\|_{L^{2}(\mathbb{R}_+, A dx)}^{2}).
\end{split}
\end{align*}
It can be obtained by requiring the following
$$
T \leq T^{\ast}:=\min\left[\left(\frac{c-1}{\|b\|_{L^{2}(0, T)}^{2}c^{p}\|u_0\|_{L^{2}(\mathbb{R}_+, A dx)}^{2p-2}}\right)^{1/3}, \, \left(\frac{c-1}{\|b\|_{L^{2}(0, T)}^{2}c^{p}\|u_1\|_{L^{2}(\mathbb{R}_+,  A dx)}^{2p-2}}\right)^{\frac{1}{3}}\right].
$$
Thus, by applying the fixed point theorem, there exists a unique local solution $u\in L^{\infty}(0, T^{\ast}; L^{2}(\mathbb{R}_+, A dx))$ of the Cauchy problem \eqref{E-WNLE}.

To prove Part (ii), we  repeat the arguments of the proof of Part (i) to get \eqref{EQ: WE-time-space-norm}. Now, by taking into account assumptions on $u_1$ and $b$ inequality \eqref{EQ: WE-time-space-norm} yields
\begin{equation}
\label{EQ: WE-time-space-norm-2}
\|u\|_{L^{\infty}(0, T; L^{2}(\mathbb{R}_+, A dx))}^{2}\leq C \Big(\|u_{0}\|_{L^{2}(\mathbb{R}_+, A dx)}^{2} + T^{3-2\gamma}  \|u\|^{2p}_{L^{\infty}(0, T; L^{2}(\mathbb{R}_+, A dx))}\Big). 
\end{equation}

For a fixed constant $c \geq 1$, let us introduce the set
$$
R_c:=\Big\{u\in L^{\infty}(0, T; L^{2}(\mathbb{R}_+, A dx)): \|u\|_{L^{\infty}(0, T; L^{2}(\mathbb{R}_+, A dx))}^{2} \leq c T^{\gamma_{0}}\|u_{0}\|_{L^{2}(\mathbb{R}_+, A dx)}^{2}\Big\},
$$
with $\gamma_{0}>0$ is to be defined later. Now, note that
\begin{align*}
\|u_{0}\|_{L^{2}(\mathbb{R}_+, A dx)}^{2} + T^{3-2\gamma}  \|u\|^{2p}_{L^{\infty}(0, T; L^{2}(\mathbb{R}_+, A dx))} 
\leq \|u_{0}\|_{L^{2}(\mathbb{R}_+, A dx)}^{2} + T^{3-2\gamma+\gamma_{0}p} c^{p} \|u_{0}\|_{L^{2}(\mathbb{R}_+, A dx)}^{2p}.    
\end{align*}

To guarantee $u\in R_c$, by invoking \eqref{EQ: WE-time-space-norm-2} we require that
\begin{align*}
\|u_{0}\|_{L^{2}(\mathbb{R}_+, A dx)}^{2} + T^{3-2\gamma+\gamma_{0}p} c^{p} \|u_{0}\|_{L^{2}(\mathbb{R}_+, A dx)}^{2p} \leq c T^{\gamma_{0}} \|u_{0}\|_{L^{2}(\mathbb{R}_+, A dx)}^{2}.   
\end{align*}
Now by choosing $0<\gamma_0<\frac{2\gamma-3}{p}$ such that
$
\tilde{\gamma}:=3-2\gamma+\gamma_{0}p<0,
$ we obtain
$$
c^{p} \|u_{0}\|_{L^{2}(\mathbb{R}_+, A dx)}^{2p-2} \leq c T^{-\tilde{\gamma}+\gamma_{0}}.
$$
From the last estimate, we conclude that for any $T>0$ there exists sufficiently small $\|u_{0}\|_{L^{2}(\mathbb{R}_+, A dx)}$ such that IVP \eqref{E-WNLE} has a solution. It proves Part (ii) of Theorem \ref{Th: E-WNLE}.
\end{proof}

%


\section*{Data Availability Statements} Data sharing not applicable to this article as no datasets were generated or analysed during the current study.

\section*{Acknowledgment}
	VK and MR are supported by FWO Odysseus 1 grant G.0H94.18N: Analysis and Partial Differential Equations and by the Methusalem programme of the Ghent University Special Research Fund (BOF)
(Grant number 01M01021). MR is also supported  by the EPSRC Grant EP/R003025/2 and by the FWO grant G022821N.

	\end{document}